\documentclass[a4paper,twoside]{article}

\usepackage[top=3cm, bottom=3cm, left=3cm, right=3cm]{geometry}

\usepackage[utf8]{inputenc}
\usepackage{bm}
\usepackage{amsfonts}
\usepackage{amsmath,amssymb,amsthm,mathtools}
\usepackage{comment}
\usepackage{enumitem}
\usepackage{cite}
\usepackage{color}
\usepackage{ulem}
\usepackage{tikz}
\newtheorem{theorem}{Theorem}[section]
\newtheorem{lemma}[theorem]{Lemma}
\newtheorem{corollary}[theorem]{Corollary}
\theoremstyle{definition}
\newtheorem{remark}[theorem]{Remark}
\newtheorem*{theorem*}{Theorem}
\providecommand{\abs}[1]{\left\lvert#1\right\rvert}

\numberwithin{equation}{section}
\counterwithin{figure}{section}

\title{Curve shortening flows on surfaces \\that are not convex at infinity}
\author{Naotoshi Fujihara}
\date{}

\begin{document}

\maketitle

\begin{abstract}
The behavior of the curve shortening flow has been extensively studied. Gage, Hamilton, and Grayson proved that, under the curve shortening flow, an embedded closed curve in the Euclidean plane becomes convex after a finite time and then shrinks to a point while remaining convex. Moreover, Grayson extended these results to surfaces that are convex at infinity and proved results similar to those for plane curves. 
In this paper, we study the curve shortening flow on surfaces that are not convex at infinity. Specifically, we consider a warped product of a unit circle and an open interval with a strictly increasing warping function. In this setting, we can define a graph property for curves within these warped products. It is known that this graph property is preserved along the curve shortening flow. Similarly to the behavior of the curve shortening flow in the plane, we prove that the curve becomes a graph after a finite time under the curve shortening flow.
\end{abstract}


\section{Introduction}\label{sec:intro}
We first introduce the curve shortening flow on the Riemannian surface $(\overline{M}, \overline{g})$. 
Let $F \colon \mathbb{S}^1 \times [0,T) \to \overline{M}$ be a smooth map and $F_t \coloneqq F(\cdot, t)$ be an embedding for each $t \in [0,T)$.
We say that $F$ is the curve shortening flow if $F$ satisfies
\begin{equation}
\frac{\partial F}{\partial t} = - \kappa_t N_t, \label{eq:CSF}
\end{equation}
where $N_t$ is the unit normal vector field of the curve $F_t$ and $\kappa_t$ is the curvature of $F_t$ with respect to $-N_t$. 
The unique existence of the solution to the initial value problem of \eqref{eq:CSF} is known, and the behavior of the curve shortening flow has been studied extensively. 
In \cite{gage1986heat}, Gage and Hamilton studied the curve shortening flow in the Euclidean plane. They proved that the smooth embedded closed convex curve remains convex and shrinks to a point in finite time under the curve shortening flow. 
In addition, Grayson showed in \cite{grayson1987heat} that any smooth embedded closed curve in the plane becomes convex after a finite time. 
Combined with these two statements, it follows that the embedded closed curve becomes convex and shrinks to a point in finite time along the curve shortening flow.
In the higher-dimensional case, Huisken \cite{Huisken} also proved that the mean curvature flow of hypersurfaces in Euclidean space preserves their convexity and that convex hypersurfaces shrink to a point. 
However, Grayson showed that his result for plane curves cannot be extended to the hypersurface case. 
There exist examples of the mean curvature flow whose initial hypersurface is not convex and the hypersurface cannot become convex or shrink to a point under the mean curvature flow. Such an example can be found in \cite{Angenent}. 

The theorem of Gage, Hamilton, and Grayson was generalized to the curve shortening flow in Riemannian surfaces.
This case was studied especially by Gage \cite{gage1990curve} and Grayson \cite{grayson1989shortening}.
Unlike the plane curve case, the curve generally does not become convex; however, it shrinks to a point under certain conditions.
Precisely, the following theorem was established by Grayson:
\begin{theorem}[Grayson \cite{grayson1989shortening}]\label{thm:Grayson}
Let $\overline{M}$ be a smooth Riemannian surface that is convex at infinity, and let $F \colon \mathbb{S}^1 \times [0,T) \to \overline{M}$ be a curve shortening flow with $F_0$ being an embedding. If $T$ is finite, then $F_t(\mathbb{S}^1)$ converges to a point. If $T$ is infinite, then the curvature of $F_t(\mathbb{S}^1)$ converges to zero in the $C^{\infty}$ norm.
\end{theorem}
A surface is convex at infinity if the convex hull of every compact subset is compact. Since a curve cannot leave a convex set under the curve shortening flow, the curve remains in a compact region. Theorem \ref{thm:Grayson} can be applied to many surfaces, in particular, closed surfaces. 
If we consider the surfaces of revolution generated by strictly increasing functions, the curve will slide off to infinity under the curve shortening flow. Hence, we cannot apply Theorem \ref{thm:Grayson} directly to the curve shortening flow on that surface. 
In this paper, we study the curve shortening flow on such surfaces, i.e., the surfaces that are not convex at infinity. 
First, we introduce the following surfaces that are not convex at infinity.  
Let $(\overline{M}, \overline{g})$ be a warped product defined by
\begin{align*}
\overline{M} \coloneqq \mathbb{S}^1 \times I, \quad \overline{g} \coloneqq r(z)^2 d \theta \otimes d \theta + d z \otimes d z,
\end{align*}
where $(\mathbb{S}^1, d \theta \otimes d \theta)$ is a unit circle, $(I \coloneqq (-\infty,a), dz \otimes dz)$ is an open interval, and $r \colon I \to \mathbb{R}$ is a smooth positive function.
Throughout this paper, we assume the following condition:
\begin{itemize}
\item[(C-0)] The warping function $r$ is strictly increasing, that is, $r'(z) > 0$ for all $z \in I$ and satisfies  
\begin{equation}\label{def:const}
C \coloneqq \sup_{z \in I} \frac{r'(z)}{{r(z)}} < \infty, \quad
D \coloneqq \sup_{z \in I} \abs{\frac{r''(z)}{{r(z)}}} < \infty.
\end{equation}
\end{itemize}
The monotonicity of the warping function $r$ implies that the ambient space $(\overline{M}, \overline{g})$ is not convex at infinity. 
In \cite{F}, the author studied the graphical curve shortening flow on this surface. 
The smooth embedded curve in $(\overline{M}, \overline{g})$ is called a graph if we have $\overline{g}(N, \partial_z) > 0$ on the curve, where $N$ denotes the unit normal vector field of the curve and $\partial_z$ is a coordinate basis of $I$. 
We define $\Theta \coloneqq \overline{g}(N, \partial_z)$ and call this function an angle function. 
Using the angle function $\Theta$, in \cite{F}, we proved that the graph property of the curve is preserved along the curve shortening flow.
In Section \ref{sec:graph}, we recall some results obtained in \cite{F} for the main theorems.
The graph property and the convexity of the plane curve are both preserved under the curve shortening flow. 
In this sense, these two properties are similar. 
In addition to this, as well as the plane curve becoming convex under the curve shortening flow, we can prove that the curve becomes a graph as in the following theorem (see also Figure \ref{fig:thm1}).
\begin{theorem}\label{thm:1}
Let $F_0 \colon \mathbb{S}^1 \to \overline{M}$ be a smooth embedded curve that is not null-homotopic, 
and let $F \colon \mathbb{S}^1\times [0,T) \to \overline{M}$ be the curve shortening flow with the initial curve $F(\cdot, 0) = F_0$. 
Then the curve shortening flow exists for all $t \in [0, \infty)$, i.e., $T = \infty$. 
Furthermore, suppose that one of the following conditions holds:
\begin{enumerate}
\item[(C-1)] $r'(z)/r(z) \to 0 $ as $z \to -\infty$;
\item[(C-2)] the length $L(0)$ of the initial curve satisfies
\begin{equation*}
L(0)^2 < \frac{1}{2 C^2}.
\end{equation*}
\end{enumerate}
Then, the curve becomes a graph at some time $T_0$, that is, $\Theta_{T_0} > 0$ along the curve shortening flow.
\end{theorem}
By condition (C-2), we obtain the following statement.
\begin{corollary}
Let $F \colon \mathbb{S}^1\times [0,\infty) \to \overline{M}$ be the curve shortening flow as in Theorem \ref{thm:1}.  
If the length $L(t)$ of the curve $F_t$ approaches $0$ as $t \to \infty$, then the curve becomes a graph after a finite time under the curve shortening flow.
\end{corollary}

The proof of Theorem \ref{thm:1} will be provided in Section \ref{sec:proof} by using Grayson's theorem for the surface case. 
By Theorem 1.1 from \cite{F} (also see Theorem \ref{thm:warped} in Section \ref{sec:graph}), we obtain the following theorem.
\begin{theorem}\label{thm:2}
Let $F_0 \colon \mathbb{S}^1 \to \overline{M}$ be a smooth embedded curve that is not null-homotopic, 
and let $F \colon \mathbb{S}^1\times [0,T) \to \overline{M}$ be the curve shortening flow with the initial curve $F(\cdot, 0) = F_0$. 
Suppose that the warping function $r$ satisfies the following conditions:
\begin{enumerate}
\item[(C-3)] $r(z) r''(z) - 2 {r'(z)}^2 \geq 0$ on $I$;
\item[(C-4)] $\sup_{z \in I}\abs{ r^{(i)}(z) / r(z) } < \infty$ for all $i \geq 1$.
\end{enumerate}
Then, the curve shortening flow exists for all $t \in [0, \infty)$, i.e., $T = \infty$, and it holds that
\begin{equation*}
\max_{p \in \mathbb{S}^1} z_t(p) \to -\infty \quad (t \to \infty),
\end{equation*}
where $z_t$ is the $z$-coordinate of the curve $F_t$ (defined by \eqref{def:z} in Section \ref{sec:graph}).
We also have 
\begin{equation*}
    \partial_{s}^{(m)} \kappa_t \to 0 \quad (t \to \infty)
\end{equation*}
for all $m \geq 0$, where $\partial_s$ is a derivative with respect to the arc length $s$.
\end{theorem}


\begin{figure}
\centering
\begin{tikzpicture}[domain=-11:-1,samples=100,>=stealth]
\draw[->,dashed] (-11.5,0) -- (-0.5,0) node[right] {$z$};

\draw plot (\x, {0.2 -1/\x});
\draw plot (\x, {-0.2 +1/\x});

\draw (-1.2,{0.2 + 1/1.2}) to [out=205,in=155] (-1.2,{-0.2 - 1/1.2});
\draw(-10.5,{0.2  +1/10.5}) to [out=180,in=180] (-10.5,{-0.2 - 1/10.5});

\draw[dotted] (-1.2,{0.2 + 1/1.2}) to [out=-40,in=40] (-1.2,{-0.2 - 1/1.2});
\draw[dotted] (-10.5,{0.2 + 1/10.5}) to [out=0,in=0] (-10.5,{-0.2 - 1/10.5});

\draw[thick] (-3,{0.2 + 1/3}) to [out=180,in=180] (-2.5,-0.4);
\draw[thick] (-2.5,-0.4) to [out=0,in=180] (-2.3,0.2);
\draw[thick] (-2.3,0.2) to [out=0,in=0] (-2.1,{-0.2 - 1/2.1});
\draw[dotted,thick] (-3,{0.2 + 1/3}) to [out=325,in=180] (-2.1,{-0.2 - 1/2.1});

\draw[->,dashed] (-3,-1.3) to (-5,-2);

\end{tikzpicture}  
\begin{tikzpicture}[domain=-11:-1,samples=100,>=stealth]
\draw[->,dashed] (-11.5,0) -- (-0.5,0) node[right] {$z$};

\draw plot (\x, {0.2 -1/\x});
\draw plot (\x, {-0.2 +1/\x});

\draw (-1.2,{0.2 + 1/1.2}) to [out=205,in=155] (-1.2,{-0.2 - 1/1.2});
\draw(-10.5,{0.2  +1/10.5}) to [out=180,in=180] (-10.5,{-0.2 - 1/10.5});

\draw[dotted] (-1.2,{0.2 + 1/1.2}) to [out=-40,in=40] (-1.2,{-0.2 - 1/1.2});
\draw[dotted] (-10.5,{0.2 + 1/10.5}) to [out=0,in=0] (-10.5,{-0.2 - 1/10.5});

\draw[thick] (-6.5,{0.2 + 1/6.2)}) to [out=180,in=190] (-6,-0.05);
\draw[thick] (-6, -0.05) to [out=10,in=0] (-5.7,{-0.2 - 1/5.7});
\draw[dotted,thick] (-6.5,{0.2 + 1/6.5}) to [out=325,in=180] (-5.7,{-0.2 - 1/5.7});

\draw[->,dashed] (-7,-1.3) to (-9,-2);
\end{tikzpicture}  
\begin{tikzpicture}[domain=-11:-1,samples=100,>=stealth]
\draw[->,dashed] (-11.5,0) -- (-0.5,0) node[right] {$z$};

\draw plot (\x, {0.2 -1/\x});
\draw plot (\x, {-0.2 +1/\x});

\draw (-1.2,{0.2 + 1/1.2}) to [out=205,in=155] (-1.2,{-0.2 - 1/1.2});
\draw(-10.5,{0.2  +1/10.5}) to [out=180,in=180] (-10.5,{-0.2 - 1/10.5});

\draw[dotted] (-1.2,{0.2 + 1/1.2}) to [out=-40,in=40] (-1.2,{-0.2 - 1/1.2});
\draw[dotted] (-10.5,{0.2 + 1/10.5}) to [out=0,in=0] (-10.5,{-0.2 - 1/10.5});

\draw[thick] (-9.5,{0.2 + 1/9.5}) to [out=180,in=0] (-9,{-0.2 - 1/9});
\draw[dotted,thick] (-9.5,{0.2 + 1/9.5}) to [out=0,in=180] (-9,{-0.2- 1/9});

\end{tikzpicture}  
\caption{Theorem \ref{thm:1}}
\label{fig:thm1}
\end{figure}

This paper is organized as follows. 
In Section \ref{sec:graph}, we explain the results from \cite{F} to introduce the concepts necessary for the main theorems. Then, we prove Theorem \ref{thm:1} and Theorem \ref{thm:2} in Section \ref{sec:proof}.

\section{Motion of graphical curves}\label{sec:graph}
In this section, we provide a brief review of the results from the author's previous paper \cite{F}.
We explain the properties of the curve shortening flow in the warped product $(\overline{M}, \overline{g})$ defined in the previous section. 
First, we compute the covariant derivatives. 
Let $\overline{\nabla}$ denote the Levi-Civita connection with respect to the Riemannian metric $\overline{g}=r(z)^2 d \theta \otimes d \theta + d z \otimes d z$. 
Define a local orthonormal frame field $\{ E_{\theta}, \partial_z \}$ by
\begin{equation*}
    E_{\theta} \coloneqq \frac{1}{r(z)} \partial_{\theta},
\end{equation*}
and $\partial_z$ is a coordinate basis of the open interval $I$.
Under this setting, we have the following formulas. 
\begin{lemma}
\label{lem:nabla bar 1}
The covariant derivatives are given by
\begin{align*}
&\overline{\nabla}_{E_{\theta}} E_{\theta} = - \frac{r'}{r}  \partial_z, \quad \overline{\nabla}_{\partial_z} E_{\theta} = 0, \\
&\overline{\nabla}_{E_{\theta}} \partial_z = \frac{r'}{r} \, E_{\theta}, \quad \overline{\nabla}_{\partial_z} \partial_z = 0. 
\end{align*}
Additionally, the Gauss curvature $\overline{K}$ is given by
\begin{equation*}
\overline{K} = -\frac{r''}{r},
\end{equation*}
hence, the Gauss curvature depends only on $z$.
\end{lemma}
The proof can be found in \cite{O'Neill}, for example.
Next, we define some functions and calculate their evolution equations.
Let $F \colon \mathbb{S}^1 \times [0,T) \to \overline{M}$ be a curve shortening flow and set $F_t \coloneqq F(\cdot, t) \colon \mathbb{S}^1 \to \overline{M}$.
The connection induced by $F$ is denoted by $\overline{\nabla}^F$.
We then define an angle function on $M$ as already mentioned in Section \ref{sec:intro}, by
\begin{equation*}
   \Theta_t \coloneqq \langle N_t, \partial_z \rangle,
\end{equation*}
and we define a height function $z_t \colon M \to I$ by
\begin{equation}\label{def:z}
z_t \coloneqq \pi_z \circ F_t,
\end{equation}
where $\pi_z$ is the projection from $\overline{M}$ to $I$.

\begin{lemma}
\label{lem:theta}
The evolution equations for $\Theta$ and $v \coloneqq \Theta^{-1}$ are as follows:
\begin{align*}
\left(\partial_t - \Delta\right) \Theta
&= \frac{r r'' - 2 {r'}^2}{r^2} \Theta (1 - \Theta^2) + \left( \frac{r'}{r} \Theta - \kappa \right)^2 \Theta, \\
\left(\partial_t - \Delta\right) v
&= -\frac{2}{v} \abs{\partial_s v}^2 -\frac{r r'' - 2 {r'}^2}{r^2}  \left( v - \frac{1}{v} \right) - \left( \frac{r'}{r} \Theta - \kappa \right)^2 v.
\end{align*}
\end{lemma}
\begin{proof}
For the arc length parameter $s$, we have $\partial_s = \partial_x / \abs{\partial_x}$, and set $\mathfrak{t} \coloneqq d F(\partial_s)$, where $x$ is a coordinate of $I$. 
Using the frame $\{ E_{\theta}, \partial_z \}$, we can express $\mathfrak{t}$ and $N$ as follows:
\begin{align*}
\mathfrak{t} &= \langle N, \partial_z \rangle E_{\theta} - \langle N, E_{\theta} \rangle \partial_z, \\
N &= \langle N, E_{\theta} \rangle E_{\theta} + \langle N, \partial_z \rangle \partial_z.
\end{align*}
We then have
\begin{align}
\partial_s \Theta 
&= \langle \overline{\nabla}^{F}_{\partial_s} N, \partial_z \rangle + \langle  N, \overline{\nabla}^{F}_{\partial_s} \partial_z \rangle \notag \\ 
&= \langle \kappa \mathfrak{t}, \partial_z \rangle + \left\langle N, \frac{r'}{r} \Theta E_{\theta} \right\rangle \notag \\
&= \left( \frac{r'}{r} \Theta - \kappa \right) \langle N, E_{\theta} \rangle, \label{dsTheta}
\end{align}
and it also holds that
\begin{equation*}
\partial_s \langle N, E_{\theta} \rangle = - \left( \frac{r'}{r} \Theta - \kappa \right) \Theta.
\end{equation*}
Hence, we obtain 
\begin{align*}
\Delta \Theta 
&= \partial_s \partial_s \Theta \\
&= \partial_s \left\{ \left( \frac{r'}{r} \Theta - \kappa \right) \langle N, E_{\theta} \rangle \right\} \\
&= \partial_s \left( \frac{r'}{r} \Theta - \kappa \right)  \langle N, E_{\theta} \rangle + \left( \frac{r'}{r} \Theta - \kappa \right) \partial_s \langle N, E_{\theta} \rangle \\
&= - \left( \frac{r'}{r} \right)' \Theta (1- \Theta^2) + \frac{r'}{r} \left( \frac{r'}{r} \Theta - \kappa \right)(1 - \Theta^2)
- \partial_s \kappa \langle N, E_{\theta} \rangle - \left( \frac{r'}{r} \Theta - \kappa \right)^2 \Theta,
\end{align*}
where we use $\partial_s z = d \pi_z (\mathfrak{t}) = - \langle N, E_{\theta} \rangle$ and $\langle N, E_{\theta} \rangle^2 = 1 - \Theta^2$.
The time derivative of $\Theta$ is given by
\begin{align*}
\partial_t \Theta 
&=  \langle \overline{\nabla}^{F}_{\partial_t} N, \partial_z \rangle + \langle N, \overline{\nabla}^{F}_{\partial_t} \partial_z \rangle \\
&=  \langle d F (\mathrm{grad}\kappa), \partial_z \rangle + \langle N, \overline{\nabla}_{- \kappa N} \partial_z \rangle \\
&= \langle (\partial_s \kappa) \mathfrak{t}, \partial_z \rangle - \kappa \langle N, E_{\theta} \rangle \langle N, \overline{\nabla}_{E_{\theta}} \partial_z \rangle \\
&= \partial_s \kappa \langle \mathfrak{t}, \partial_z \rangle - \kappa \langle N, E_{\theta} \rangle \left\langle N, \frac{r'}{r} E_{\theta} \right\rangle \\
&= - \partial_s \kappa \langle N, E_{\theta} \rangle  - \kappa \frac{r'}{r} (1 - \Theta^2).
\end{align*}
Thus, it follows that 
\begin{align*}
(\partial_t - \Delta) \Theta
&= - \partial_s \kappa \langle N, E_{\theta} \rangle - \kappa \frac{r'}{r} (1 - \Theta^2) \\
&+ \partial_s \kappa \langle N, E_{\theta} \rangle + \frac{r r'' - {r'}^2}{r^2} \Theta (1 - \Theta^2) 
 - \frac{r'}{r} \left( \frac{r'}{r} \Theta - \kappa \right)(1-\Theta^2) + \left( \frac{r'}{r} \Theta - \kappa \right)^2 \Theta \\
& = \frac{r r'' - 2 {r'}^2}{r^2} \Theta (1 - \Theta^2) +  \left( \frac{r'}{r} \Theta - \kappa \right)^2 \Theta.
\end{align*}
The evolution equation for $v = \Theta^{-1}$ follows from the evolution equation for $\Theta$.
\end{proof}
By applying the maximal principle to the evolution equation for the function $v$, we obtain the graph-preserving property of the curve shortening flow.
\begin{lemma}
\label{thm:graph}
Let $F \colon \mathbb{S}^1 \times [0,T) \to \overline{M}$ be a curve shortening flow with $F_0$ being a graph. Then, for the warping function $r$ satisfying (C-0), the curve $F_t$ remains a graph for all $t \in [0,T)$.
Furthermore, if we assume condition (C-3), then $v_t$ is uniformly bounded.
\end{lemma}

At the end of this section, we recall Theorem 1.1 in \cite{F} for the main theorems.
\begin{theorem}[\cite{F}]\label{thm:warped}
Let $(\overline{M}, \overline{g}= r(z)^2 d \theta \otimes d \theta + d z \otimes d z)$ be a warped product of a unit circle $(\mathbb{S}^1, d \theta \otimes d \theta)$ and an open interval $(I, dz \otimes dz)$ with a smooth positive function $r \colon I \to \mathbb{R}$ satisfying (C-0). 
Suppose that $F \colon \mathbb{S}^1 \times [0,T) \to \overline{M}$ is a curve shortening flow with $F_0$ being a graph. 
Then the following statements hold:
\begin{itemize}
\item[(i)] The curve $F_t$ remains a graph for all $t \in [0,T)$;
\item[(ii)] The curve shortening flow $\{ F_t \}_t$ exists for all $t \in [0,\infty)$, i.e., $T = \infty$. 
\item[(iii)] If conditions (C-3) and  (C-4) hold, then we have 
\begin{equation*}
    \partial_{s}^{(m)} \kappa_t \to 0 \quad (t \to \infty)
\end{equation*}
for all $m \geq 0$.
\end{itemize}
\end{theorem}
From this theorem, we obtain the long-time existence of the graphical curve shortening flow. 
However, we note that the long-time existence of the graphical curve shortening flow follows from Theorem \ref{thm:Grayson} independently.

\section{Proofs of the main theorems}\label{sec:proof}
We prove the main theorems in this section. Thanks to Grayson's theorem, the long-time existence of the curve shortening flow is guaranteed without the assumption that the initial curve is a graph (Lemma \ref{lem:longtime}). 
Note here that we always assume the condition (C-0) in this section. 
Now, we define $Z \colon M \times [0,T) \to \overline{M}$ by $Z(p,t) \coloneqq (p, z(t))$. 
The geodesic curvature of the curve $Z_t \coloneqq Z(\cdot, t)$ is given by $r'(z(t)) / r(z(t))$, and the equation \eqref{eq:CSF} is equivalent to the following equation:
\begin{equation}
\label{eq:standard}
    \frac{d z}{d t}(t) = - \frac{r'(z(t))}{r(z(t))}.
\end{equation}
The solution to \eqref{eq:standard} is used to control the motion of the curve shortening flow. 
First, we note that the curve shortening flow $Z$ obtained from \eqref{eq:standard} can be defined for all $t \in [0,\infty)$ by the assumption (C-0). 
Let $\{ Z_t \}_{t \in [0, \infty)}$ and $\{ \widetilde{Z}_t \}_{t \in [0, \infty)}$ be the curve shortening flow obtained from \eqref{eq:standard}.
Suppose also that $\{ F_t \}_{t \in [0,T)}$ is the curve shortening flow that satisfies $z(0) < \min_{p \in \mathbb{S}^1}z_0(p)$ and $\max_{p \in \mathbb{S}^1} z_0(p) < \widetilde{z}(0)$. 
Then, by the comparison principle of the curve shortening flow, we have
\begin{equation}\label{ineq:comp}
z(t) < \min_{p \in \mathbb{S}^1} z_t(p), \quad \max_{p \in \mathbb{S}^1} z_t(p) < \widetilde{z}(t)
\end{equation}
for all $t \in [0,T)$ (see Figure \ref{fig:comparison}). 
Hence, we obtain the following lemma.

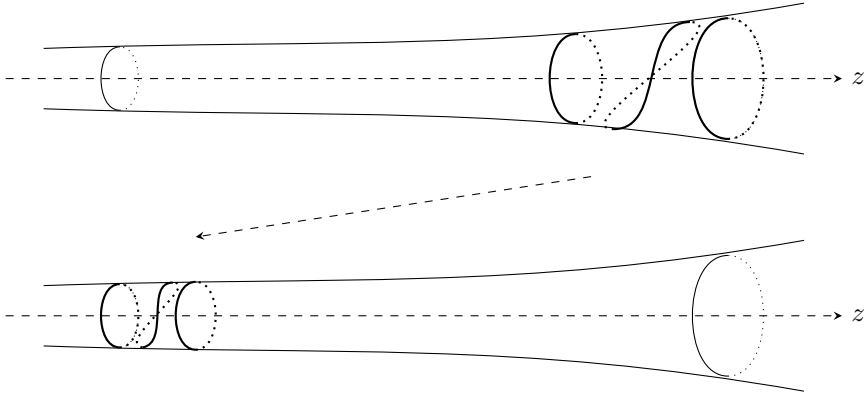
\begin{figure}
\centering
\begin{tikzpicture}[domain=-5:5,samples=100,>=stealth]
\draw[->,dashed] (-5.5,0) -- (5.5,0) node[right] {$z$};

\draw(-5,0.4)to[out=2,in=-170](5,1);
\draw(-5,-0.4)to[out=-2,in=170](5,-1);

\draw (4,-0.8) to [out=180,in=180] (4,0.8);
\draw[dotted] (4,-0.8) to [out=0,in=0] (4,0.8);

\draw (-4,-0.42) to [out=180,in=180] (-4,0.42);
\draw[dotted] (-4,-0.42) to [out=0,in=0] (-4,0.42);

\draw[thick] (4,-0.8) to [out=180,in=180] (4,0.8);
\draw[dotted,thick] (4,-0.8) to [out=0,in=0] (4,0.8);

\draw[thick] (2.5,-0.67) to [out=0,in=180] (3.5,0.75); 
\draw[dotted,thick] (2.5,-0.67) to [out=180,in=0] (3.5,0.75); 

\draw[thick] (2,-0.59) to [out=180,in=180] (2,0.59);
\draw[dotted,thick] (2,-0.59) to [out=0,in=0] (2,0.59);

\draw[->,dashed] (2.2,-1.3) to (-3,-2.1);

\end{tikzpicture}
\begin{tikzpicture}[domain=-5:5,samples=100,>=stealth]
\draw[->,dashed] (-5.5,0) -- (5.5,0) node[right] {$z$};

\draw(-5,0.4)to[out=2,in=-170](5,1);
\draw(-5,-0.4)to[out=-2,in=170](5,-1);

\draw (4,-0.8) to [out=180,in=180] (4,0.8);
\draw[dotted] (4,-0.8) to [out=0,in=0] (4,0.8);

\draw (-4,-0.42) to [out=180,in=180] (-4,0.42);
\draw[dotted] (-4,-0.42) to [out=0,in=0] (-4,0.42);

\draw[thick] (-3,-0.45) to [out=180,in=180] (-3,0.45);
\draw[dotted,thick] (-3,-0.45) to [out=0,in=0] (-3,0.45);

\draw[thick] (-3.7,-0.42) to [out=10,in=180] (-3.3,0.44); 
\draw[dotted,thick] (-3.7,-0.42) to [out=170,in=0] (-3.3,0.44); 

\draw[thick] (-4,-0.42) to [out=180,in=180] (-4,0.42);
\draw[dotted,thick] (-4,-0.42) to [out=0,in=0] (-4,0.42);

\end{tikzpicture}
\caption{Comparison principle}
\label{fig:comparison}
\end{figure}

\begin{lemma}
\label{lem:longtime}
Let $F_0 \colon \mathbb{S}^1 \to \overline{M}$ be a smooth embedded curve that is not null-homotopic. 
Then, the curve shortening flow $\{ F_t \}_t$ with the initial curve $F_0$ exists for all $t \in [0,\infty)$. 
Furthermore, we have
\begin{equation*}
\max_{p \in \mathbb{S}^1} z_t(p) \to -\infty \quad (t \to \infty).
\end{equation*}
\end{lemma}
\begin{proof}
Suppose that the maximal time of existence $T$ is finite; then by the above comparison principle argument, the curve shortening flow $\{ F_t \}_{t \in [0,T)}$ must lie within a compact region of $\overline{M}$. 
Hence, by Remark \ref{rem:convex at infinity}, we can apply Theorem \ref{thm:Grayson} in this situation. 
However, since the curve $F_0$ is not null-homotopic, the curve $F_t$ cannot shrink to a point. This leads to a contradiction. 
Thus, the curve shortening flow $\{ F_t \}_{t}$ with the initial curve $F_0$ exists for all $t \in [0,\infty)$. 
The second statement follows from \eqref{ineq:comp} since $\widetilde{z}(t) \to -\infty$ as $t \to \infty$ holds for the curve shortening flow $\widetilde{Z}(p,t) = (p,\widetilde{z}(t))$ obtained from \eqref{eq:standard}. 
\end{proof}
\begin{remark}\label{rem:convex at infinity}
The bounded subset of the warped product $(\overline{M}, \overline{g})$ can be isometrically embedded into a surface that is convex at infinity. 
Suppose that the subset is contained in $\mathbb{S}^1 \times (a_0, a)$. 
Then, define a smooth positive function $\widetilde{r} \colon (-\infty, a) \to (0, \infty)$ by
\begin{equation*}
\widetilde{r}(z) \coloneqq
\begin{cases}
r(z) + (z-a_0)^2 e^{1/(z-a_0)}  &\text{for }z \in  (-\infty, a_0) \\
r(z) &\text{for }z \in [a_0, a).
\end{cases}
\end{equation*}
By this definition, there exists a number $a_1 \in (-\infty, a_0)$ such that $\widetilde{r}'(z) < 0$ for all $z \in (-\infty, a_1)$. 
Hence, the warped product of the unit circle $\mathbb{S}^1$ and the open interval $(-\infty, a)$ with the warping function $\widetilde{r}$ is convex at infinity (see Figure \ref{fig:warping}). Therefore, we can apply Theorem \ref{thm:Grayson} to the curve shortening flow in the bounded subset.
\begin{figure}
\centering
\begin{tikzpicture}[domain=-11:-1,samples=100,>=stealth]
\draw[->,dashed] (-11.5,0) -- (-0.5,0) node[right] {$z$};

\draw plot (\x, {0.2 -1/\x});
\draw[domain=-11:-6] plot (\x, {0.2 -1/\x + pow(\x + 6, 2)/20});

\draw (-11, 0.3) node[left] {$r(z)$};
\draw (-11, 1.6) node[left] {$\widetilde{r}(z)$};
\draw[dotted] (-6, 0.35) -- (-6,0) node[below] {$a_0$};
\end{tikzpicture}
\caption{Construction of the warping function}
\label{fig:warping}
\end{figure}
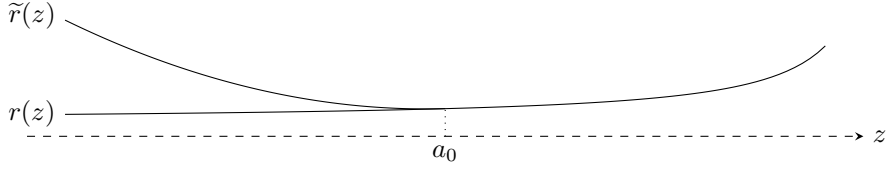
\end{remark}

From Lemma \ref{lem:longtime}, we can consider the curve shortening flow $\{ F_t \}_{t \in [0,\infty)}$ defined for all time, provided that the initial curve is not null-homotopic. 
Therefore, throughout the rest of this section, we assume such a curve shortening flow. 
The following lemma is essential for the main theorem. The assumption on the angle function means that the curve is not a graph for all time along the curve shortening flow when we take $c_0 = 0$.
\begin{lemma}\label{lem:notgraph}
Let $0 \leq c_0 < 1$ be a constant and suppose that 
\begin{equation*}
\min_{p \in \mathbb{S}^1} \Theta_t (p) \leq c_0
\end{equation*}
for all $t \in [0,\infty)$. 
Then, we have
\begin{equation*}
\frac{1}{2 L(t)^2} - \frac{\max_{p \in \mathbb{S}^1} {\kappa_t(p)}^2}{{d_0}^2} \leq \frac{1}{{d_0}^2} \max_{p \in \mathbb{S}^1}\left( \frac{r'(z_t(p))}{r(z_t(p))} \right)^2 ,
\end{equation*}
where $d_0 \coloneqq 1 - c_0 > 0$.
\end{lemma}

\begin{proof}
Let us take points $p_t$, $q_t \in \mathbb{S}^1$ so that they satisfy $\Theta_t(p_t) = 1$ and $\Theta_t(q_t) = \min_{p \in \mathbb{S}^1} \Theta_t (p)$, then we have
\begin{align*}
d_0
= 1 - c_0 
\leq \abs{\Theta_t(p_t) - \Theta_t(q_t)} 
\leq \abs{\int_{q_t}^{p_t} \partial_s \Theta_t d s} 
\leq \int \abs{\partial_s \Theta_t} d s.
\end{align*}
From \eqref{dsTheta}, the derivative of the angle function is given by
\begin{equation*}
\partial_s \Theta_t = \left( \frac{r'(z_t)}{r(z_t)} \Theta_t - \kappa_t \right) \langle N_t, E_{\theta} \rangle.
\end{equation*}
Hence, it follows that 
\begin{align*}
d_0
&\leq \int \abs{\partial_s \Theta_t} d s \\
&= \int \abs{\frac{r'(z_t)}{r(z_t)} \Theta_t - \kappa_t} \abs{\langle N_t, E_{\theta} \rangle} d s \\
&\leq \int \abs{\frac{r'(z_t)}{r(z_t)}} d s + \int \abs{\kappa_t} d s,
\end{align*}
since $\abs{\Theta_t} \leq 1$ and $\abs{\langle N_t, E_{\theta} \rangle} \leq 1$. 
Then, we have
\begin{align}
{d_0}^2 
&\leq \left( \int \abs{\frac{r'(z_t)}{r(z_t)}} d s + \int \abs{\kappa_t} d s \right)^2 \notag\\
&\leq 2 \left( \int \abs{\frac{r'(z_t)}{r(z_t)}} d s \right)^2 + 2 \left( \int \abs{\kappa_t} d s \right)^2  \label{ineq}\\
&\leq 2 L(t)^2 \max_{p \in \mathbb{S}^1}\left( \frac{r'(z_t(p))}{r(z_t(p))} \right)^2  + 2 L(t)^2 \max_{p \in \mathbb{S}^1} {\kappa_t (p)}^2. \notag
\end{align}
Thus, we obtain the desired inequality:
\begin{equation*}
\frac{1}{2 L(t)^2} - \frac{\max_{p \in \mathbb{S}^1} {\kappa_t(p)}^2}{{d_0}^2} \leq \frac{1}{{d_0}^2} \max_{p \in \mathbb{S}^1}\left( \frac{r'(z_t(p))}{r(z_t(p))} \right)^2 .
\end{equation*}
\end{proof}

The calculation of the following lemma is based on \cite{gage1990curve} and \cite{grayson1989shortening}. 
However, the key difference in our case is that the length $L(t)$ of the curve may approach zero as $t \to \infty$, which requires certain modifications.
\begin{lemma}\label{lem:curvature}
Let $0 < d_0 \leq 1$ and $A_0 \geq 0$ be some constants. If the inequality
\begin{equation*}
\frac{1}{2 L(t)^2} - \frac{\max_{p \in \mathbb{S}^1} {\kappa_t(p)}^2}{{d_0}^2} \leq A_0
\end{equation*}
holds for all $t \in [0,\infty)$, then we have
\begin{equation*}
L(t) \int \kappa_t^2 d s \to 0 \quad (t \to \infty).
\end{equation*}
\end{lemma}

\begin{proof}
Let $\psi$ be a function defined by $\psi(t) \coloneqq L(t) \int \kappa_t^2 ds$. 
The derivative of the length function $L(t)$ is given by 
\begin{equation*}
\frac{d}{d t} L(t) = \frac{d}{d t} \int ds = - \int \kappa_t^2 ds,
\end{equation*}
so we obtain
\begin{equation*}
\int_0^{\infty} \int \kappa^2 ds dt \leq L(0).
\end{equation*}
This implies that we can take a sequence $\{ t_n \}_{n \in \mathbb{N}}$ such that $\int \kappa_t^2 ds \big\vert_{t_n} \to 0$. 
We also have $\psi(t_n) \to 0$. 
We prove $\psi(t) \to 0$ as $t \to \infty$. 
For any positive constant $\varepsilon$ that satisfies
\begin{equation}\label{ineq:epsilon}
\varepsilon <  \min \left\{ \frac{{d_0}^2}{4+{d_0}^2}, \frac{1}{2L(0)(4 A_0 + 2 D)} \right\},
\end{equation}
there is a time $T_1$ such that 
\begin{align}
\psi(T_1) &< \frac{\varepsilon}{2}, \label{ineq:psi}\\
\int_{T_1}^{\infty} \int \kappa^2 ds dt &< \varepsilon^2. \label{ineq:integral}
\end{align}
We prove by contradiction that $\psi(t) < \varepsilon$ holds for all $t > T_1$. 
Suppose this is not the case. 
Then, there exists a time $T_2 > T_1$ such that $\psi(T_2) = \varepsilon$ for the first time. 
Now, we compute the derivative of $\psi(t)$:
\begin{align*}
\frac{d}{d t} \psi(t)
&= - \left( \int \kappa_t^2 ds \right)^2 + L(t) \int \frac{\partial}{\partial t} \kappa^2 - \kappa_t^4 ds \\
&=  - \left( \int \kappa_t^2 ds \right)^2 + L(t) \int \Delta \kappa_t^2 -2 \abs{\partial_s \kappa_t}^2 + 2 \kappa_t^4 - 2\kappa_t^2 \frac{r''}{r} - \kappa_t^4 ds \\
&\leq -2 L(t) \int \abs{\partial_s \kappa_t}^2 ds  +  L(t) \max_{p \in \mathbb{S}^1} \kappa_t(p)^2 \int \kappa_t^2 ds + 2 D L(0)\int \kappa_t^2 ds,
\end{align*}
where we use $\abs{r''(z) / r(z)} \leq D $ and the evolution equation for the curvature $\kappa$:
\begin{equation*}
(\partial_t - \Delta) \kappa^2 = -2 \abs{\partial_s \kappa}^2 + 2 \kappa^2 \left( \kappa^2 +\overline{K} \right). 
\end{equation*}
Here, we have $\abs{\kappa_t(p) - \kappa_t(q)} \leq \int \abs{\partial_s \kappa_t} d s$ for all $p, q \in \mathbb{S}^1$, hence we obtain
\begin{equation*}
\kappa_t(p)^2  
\leq \left( \abs{\kappa_t(q)} + \int \abs{\partial_s \kappa_t} d s \right)^2 
\leq 2 \kappa_t(q)^2 + 2 \left( \int \abs{\partial_s \kappa_t} ds  \right)^2.
\end{equation*}
We also have $L(t) \min_{q \in \mathbb{S}^1} \kappa_t(q)^2 \leq \int \kappa_t^2 d s$. Thus, we conclude that
\begin{equation*}
\max_{p \in \mathbb{S}^1} \kappa_t(p)^2 
\leq 2 \min_{q \in \mathbb{S}^1} \kappa_t(q)^2 + 2 \left( \int \abs{\partial_s \kappa_t} ds  \right)^2 
\leq \frac{2}{L(t)} \int \kappa_t^2 ds + 2 L(t) \int \abs{\partial_s \kappa_t}^2 ds.
\end{equation*}
From this inequality, it follows that
\begin{equation*}
- 2 L(t) \int \abs{\partial_s \kappa_t}^2 ds \leq \frac{2}{L(t)} \int \kappa_t^2 ds - \max_{p \in \mathbb{S}^1} \kappa_t(p)^2.
\end{equation*}
Thus, we obtain
\begin{align*}
\frac{d}{d t} \psi(t) 
&\leq  \frac{2}{L(t)} \int \kappa_t^2 ds - \max_{p \in \mathbb{S}^1} \kappa_t(p)^2 +  L(t) \max_{p \in \mathbb{S}^1} \kappa_t(p)^2 \int \kappa_t^2 ds + 2 D L(0)\int \kappa_t^2 ds \\
&=\frac{2}{L(t)} \int \kappa_t^2 ds  - \max_{p \in \mathbb{S}^1} \kappa_t(p)^2  + \psi(t) \max_{p \in \mathbb{S}^1} \kappa_t(p)^2 + 2 D L(0)\int \kappa_t^2 ds.
\end{align*}
For $T_1 \leq t \leq T_2$, we have 
\begin{equation*}
0 \leq \psi(t) \leq \varepsilon < \frac{{d_0}^2}{4+{d_0}^2},
\end{equation*}
hence, it follows that
\begin{equation*}
- \max_{p \in \mathbb{S}^1} \kappa_t(p)^2 \leq - \frac{4+{d_0}^2}{{d_0}^2} \psi(t) \max_{p \in \mathbb{S}^1} \kappa_t(p)^2.
\end{equation*}
Therefore, we have the following inequality:
\begin{equation}\label{ineq:dpsi}
\begin{aligned}
\frac{d}{d t} \psi(t) 
&\leq  \frac{2}{L(t)} \int \kappa_t^2 ds - \frac{4+{d_0}^2}{{d_0}^2} \psi(t) \max_{p \in \mathbb{S}^1} \kappa_t(p)^2 + \psi(t) \max_{p \in \mathbb{S}^1} \kappa_t(p)^2 + 2 D L(0)\int \kappa_t^2 ds \\
&=\frac{2}{L(t)^2}\psi(t) - \frac{4}{{d_0}^2} \psi(t) \max_{p \in \mathbb{S}^1} \kappa_t(p)^2  + 2 D L(0)\int \kappa_t^2 ds \\
&\leq 4 \left( \frac{1}{2 L(t)^2} -  \frac{\max_{p \in \mathbb{S}^1} \kappa_t(p)^2}{{d_0}^2} \right)  \psi(t)  + 2 D L(0)\int \kappa_t^2 ds \\
&\leq (4 A_0 + 2 D ) L(0) \int \kappa_t^2 ds,
\end{aligned}
\end{equation}
where we use the monotonicity of the length function $L(t)$ and the assumption stated in the lemma:
\begin{equation*}
\frac{1}{2 L(t)^2} - \frac{\max_{p \in \mathbb{S}^1} {\kappa_t(p)}^2}{{d_0}^2} \leq A_0.
\end{equation*}
From \eqref{ineq:epsilon}, \eqref{ineq:psi}, and \eqref{ineq:integral}, we integrate the inequality \eqref{ineq:dpsi} to obtain
\begin{align*}
\frac{\varepsilon}{2} 
&< \psi(T_2) - \psi(T_1) \\
&\leq  (4 A_0 + 2 D ) L(0) \int_{T_1}^{T_2} \int \kappa^2 ds dt \\
&\leq \frac{1}{2\varepsilon}  \int_{T_1}^{\infty} \int \kappa^2 ds dt \\
&\leq \frac{\varepsilon}{2}.
\end{align*}
This is a contradiction, and hence, we have
\begin{equation*}
\psi(t) < \varepsilon
\end{equation*}
for all $t > T_1$. 
Thus, it follows that $\psi(t) \to 0$ as $t \to \infty$.
\end{proof}

Finally, we are prepared to present the proofs of the main theorems.

\begin{proof}[Proof of Theorem \ref{thm:1} and Theorem \ref{thm:2}]
Suppose that the curve shortening flow $\{ F_t \}_{t \in [0,\infty)}$ is not a graph for all time, that is, suppose that
\begin{equation*}
\min_{p \in \mathbb{S}^1} \Theta_t(p) \leq 0
\end{equation*}
for all $t \in [0,\infty)$. 
By Lemma \ref{lem:notgraph} and Lemma \ref{lem:curvature}, we have $L(t) \int \kappa^2 d s \to 0$.
From the inequality \eqref{ineq} with $d_0 = 1$, we have
\begin{equation}\label{ineq2}
1 
\leq 2 \left( \int \abs{\frac{r'(z_t)}{r(z_t)}} d s \right)^2 + 2 \left( \int \abs{\kappa_t} d s \right)^2 
\leq 2 L(t)^2 \max_{p \in \mathbb{S}^1}\left( \frac{r'(z_t(p))}{r(z_t(p))} \right)^2 + 2 L(t) \int \kappa_t^2 d s.
\end{equation}
If the condition (C-1) holds, then the right-hand side of \eqref{ineq2} approaches zero as $t \to \infty$, since $\max_{p \in \mathbb{S}^1} z_t(p) \to -\infty$ holds from Lemma \ref{lem:longtime}. This leads to a contradiction. 
Likewise, from \eqref{def:const} and \eqref{ineq2}, it follows that 
\begin{equation*}
1 
\leq 2 L(t)^2 \max_{p \in \mathbb{S}^1}\left( \frac{r'(z_t(p))}{r(z_t(p))} \right)^2 + 2 L(t) \int \kappa_t^2 d s 
\leq 2 L(0)^2 C^2 + 2 L(t) \int \kappa_t^2 d s.
\end{equation*}
Taking the limit, we obtain
\begin{equation*}
1 \leq 2 L(0)^2 C^2.
\end{equation*}
Hence, if the condition (C-2) holds, that is, the length of the initial curve satisfies
\begin{equation*}
L(0)^2 < \frac{1}{2 C^2},
\end{equation*}
this leads to a contradiction. Thus, there exists some time $T_0$ such that 
\begin{equation*}
\min_{p \in \mathbb{S}^1} \Theta_{T_0}(p) > 0,
\end{equation*}
which means that the curve $F_{T_0}$ is a graph. This proves Theorem \ref{thm:1}. 
Finally, Theorem \ref{thm:2} follows directly from Theorem \ref{thm:1}, Theorem \ref{thm:warped}, and Lemma \ref{lem:longtime}.
\end{proof}

\begin{remark}
We would like to make a remark on Theorem \ref{thm:1}. 
In this paper, we assume that the warping function $r$ is strictly increasing. 
If there are points where $r'$ vanish, we obtain the same result as Theorem \ref{thm:1} by using Lemma A.2 from Gage's paper \cite{gage1990curve}. 
The lemma roughly states that if a closed curve is sufficiently close to a certain closed geodesic and its geodesic curvature is sufficiently small, then the curve can be described non-parametrically over the geodesic. 
To apply the lemma, we assume that the warping function $r \colon (-a, a) \to (0, \infty)$ satisfies that
$r'(0) = 0$, $r'(z) > 0$ for $0 < z < a$, and $r'(z) < 0$ for $-a < z < 0$ for simplicity.
Under these conditions, the curve $\gamma \colon \mathbb{S}^1 \to \overline{M}$ defined by $\gamma(p) \coloneqq (p, 0)$ is a geodesic. 
By the comparison principle, we conclude that the curve approaches the geodesic $\gamma$ under the curve shortening flow. 
Furthermore, the curvature of the curve converges to zero by Theorem \ref{thm:Grayson}. 
Thus, from Lemma A.2 in \cite{gage1990curve}, the curve becomes a graph at a certain time along the curve shortening flow (see Figure \ref{fig:gage}).

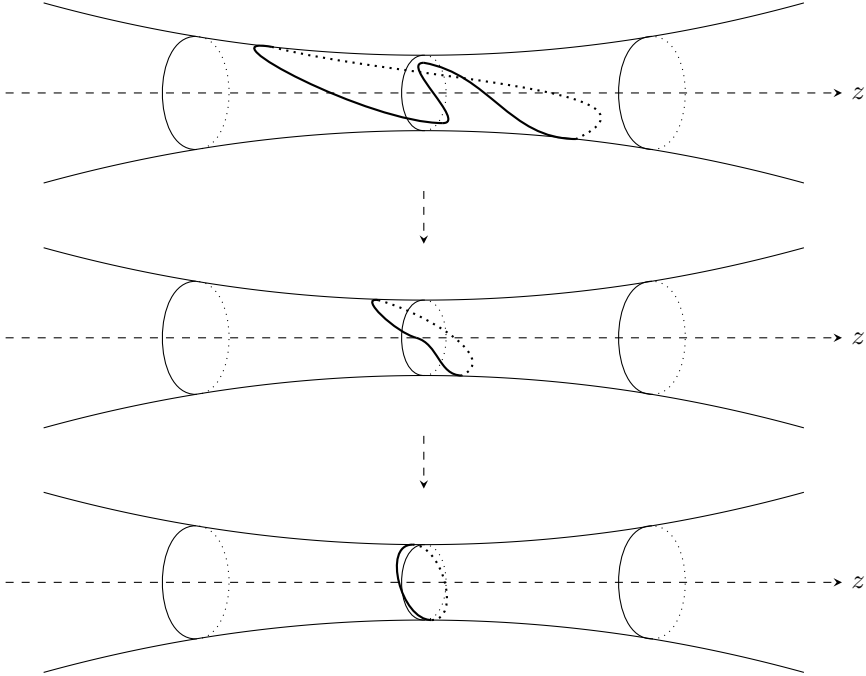
\begin{figure}
\centering
\begin{tikzpicture}[domain=-5:5,samples=100,>=stealth]
\draw[->,dashed] (-5.5,0) -- (5.5,0) node[right] {$z$};

\draw plot (\x, {pow(\x,2)/36 + 0.5});
\draw plot (\x, {-pow(\x,2)/36 - 0.5});

\draw (-3,-0.75) to [out=180,in=180] (-3,0.75);
\draw[dotted] (-3,-0.75) to [out=0,in=0] (-3,0.75);

\draw (3,-0.75) to [out=180,in=180] (3,0.75);
\draw[dotted] (3,-0.75) to [out=0,in=0] (3,0.75);

\draw (0,-0.5) to [out=180,in=180] (0,0.5);
\draw[dotted] (0,-0.5) to [out=0,in=0] (0,0.5);

\draw[thick] (-2,1/9 +0.5) to [out=170,in=180] (0.25,-0.4);
\draw[thick] (0.25,-0.4) to [out=0,in=180] (0,0.4);
\draw[thick] (0,0.4) to [out=-10,in=180] (2,-1/9 -0.5);
\draw[dotted,thick] (-2,1/9 +0.5) to [out=-15,in=25] (2,-1/9 -0.5);

\draw[->, dashed] (0,-1.3) to (0,-2);

\end{tikzpicture}
\begin{tikzpicture}[domain=-5:5,samples=100,>=stealth]
\draw[->,dashed] (-5.5,0) -- (5.5,0) node[right] {$z$};

\draw plot (\x, {pow(\x,2)/36 + 0.5});
\draw plot (\x, {-pow(\x,2)/36 - 0.5});

\draw (-3,-0.75) to [out=180,in=180] (-3,0.75);
\draw[dotted] (-3,-0.75) to [out=0,in=0] (-3,0.75);

\draw (3,-0.75) to [out=180,in=180] (3,0.75);
\draw[dotted] (3,-0.75) to [out=0,in=0] (3,0.75);

\draw (0,-0.5) to [out=180,in=180] (0,0.5);
\draw[dotted] (0,-0.5) to [out=0,in=0] (0,0.5);

\draw[thick] (-0.6,0.5) to [out=170,in=165] (-0.1,0);
\draw[thick] (-0.1,0) to [out=-15,in=180] (0.5,-0.5);
\draw[dotted,thick] (-0.6,0.5) to [out=-15,in=25] (0.5,-0.5);

\draw[->, dashed] (0,-1.3) to (0,-2);
\end{tikzpicture}
\begin{tikzpicture}[domain=-5:5,samples=100,>=stealth]
\draw[->,dashed] (-5.5,0) -- (5.5,0) node[right] {$z$};

\draw plot (\x, {pow(\x,2)/36 + 0.5});
\draw plot (\x, {-pow(\x,2)/36 - 0.5});

\draw (-3,-0.75) to [out=180,in=180] (-3,0.75);
\draw[dotted] (-3,-0.75) to [out=0,in=0] (-3,0.75);

\draw (3,-0.75) to [out=180,in=180] (3,0.75);
\draw[dotted] (3,-0.75) to [out=0,in=0] (3,0.75);

\draw (0,-0.5) to [out=180,in=180] (0,0.5);
\draw[dotted] (0,-0.5) to [out=0,in=0] (0,0.5);

\draw[thick] (-0.15,0.5) to [out=180,in=180] (0.1,-0.5); 
\draw[dotted,thick] (-0.15,0.5) to [out=0,in=0] (0.1,-0.5); 

\end{tikzpicture}
\caption{Curve shortening flow on the surface that is convex at infinity}
\label{fig:gage}
\end{figure}

\end{remark}

\section*{Acknowledgement} 
I would like to express my gratitude to my supervisor, Naoyuki Koike, for his invaluable advice and consistent support throughout my research.


\vspace{0.5truecm}

\begin{flushright}
{\small 
Naotoshi Fujihara \\
Department of Mathematics \\
Graduate School of Science \\
Tokyo University of Science \\
1-3 Kagurazaka \\
Shinjuku-ku \\
Tokyo 162-8601 \\
Japan \\
(E-mail: 1123706@ed.tus.ac.jp)
}
\end{flushright}

\end{document}